\documentclass{amsart}

\usepackage{tikz}
\usetikzlibrary{matrix}

\usepackage{amsmath}
\usepackage{listings}
\usepackage{amssymb}
\usepackage{hyperref}
\usepackage{mathtools}
\usepackage{amsthm}
\usepackage{mathrsfs}
\usepackage{xcolor}

\usepackage{graphicx}

\usepackage[shortlabels]{enumitem}

\usepackage[T2A]{fontenc}
\usepackage[utf8]{inputenc}

\newtheorem{thm}{Theorem}[section]

\newtheorem{defn}[thm]{Definition}
\newtheorem{prop}[thm]{Proposition}
\newtheorem{lemma}[thm]{Lemma}

\newtheoremstyle{boldremark}
{\dimexpr\topsep/2\relax} 
{\dimexpr\topsep/2\relax} 
{}          
{}          
{\bfseries} 
{.}         
{.5em}      
{}          

\theoremstyle{boldremark}
\newtheorem{rem}[thm]{Remark}
\newenvironment{remark}
{\pushQED{\qed}\rem}
{\popQED\endrem}

\newtheorem{examplex}[thm]{Example}
\newenvironment{example}
{\pushQED{\qed}\examplex}
{\popQED\endexamplex}

\newcommand{\N}{{\ensuremath{\mathbb{N}}}}

\newcommand{\R}{{\ensuremath{\mathbb{R}}}}

\newcommand{\muzeros}{{\ensuremath{\mu_{\textbf{0}}}}}
\newcommand{\nuzeros}{{\ensuremath{\nu_{\textbf{0}}}}}

\newcommand{\muzerosn}{{\ensuremath{\mu^n_{\textbf{0}}}}}
\newcommand{\nuzerosn}{{\ensuremath{\nu^n_{\textbf{0}}}}}

\newcommand{\supp}{\ensuremath{\textrm{supp}}}

\usepackage{pdfsync}
\title{On a discrete max-plus transportation problem}

\author{Pedro Barrios}
\address [Pedro Barrios]{Universidad de Antioquia,
		Calle 67, No. 53-108,
		Medell{í}n.
		Colombia}
\email[Pedro Barrios]{pdrlsbrrs@gmail.com}

\author{Sergio Mayorga}
\address[Sergio Mayorga]{
Innopolis University, 
Ul. Universitetskaya 1,
Innopolis,
Russian Federation
}
\email[Sergio Mayorga]{me@mayorga.ru}

\author{Eugene Stepanov}
\thanks{The third author acknowledges the MIUR Excellence Department Project awarded to the Department of Mathematics, University of Pisa, CUP I57G22000700001}
\dedicatory{Dedicated to N.N. Uraltseva on the occasion of her 90th birthday}
\address[Eugene Stepanov]{St.Petersburg Branch of the Steklov Mathematical Institute of the Russian Academy of Sciences,
	St.Petersburg, 
	Russian Federation
	\and
	Dipartimento di Matematica, Universit\`a di Pisa,
	Largo Bruno Pontecorvo 5, 56127 Pisa, Italy
	\and
	HSE University, Moscow, Russian Federation		
}
\email{stepanov.eugene@gmail.com}

\date{}
\begin{document}
	
	\begin{abstract}
		We provide an explicit algorithm to solve
		the idempotent analogue of the discrete
		Monge-Kantorovich optimal mass transportation
		problem with the usual real number field replaced
		by the tropical (max-plus) semiring,
		in which addition is defined as the maximum and
		product is defined as usual addition, with
		$-\infty$ and $0$ playing the roles
		of additive and multiplicative identities.
		Such a problem may be naturally called tropical
		or ``max-plus'' optimal transportation problem.
		We show that the solutions to the latter,
		called the optimal tropical plans, may not correspond to
		perfect matchings even if the data (max-plus probability measures)
		have all weights equal to zero,
		in contrast with the classical discrete optimal transportation analogue,
		where perfect matching optimal plans in similar situations always exist.
		Nevertheless, in some randomized situation the
		existence of
		perfect matching optimal tropical plans may occur rather frequently.
		At last, we prove that the uniqueness of solutions of the optimal
		tropical transportation problem is quite rare.    
	\end{abstract}
	
	
	\maketitle
	
	\section{Introduction}
	
	In this paper we consider a discrete optimization problem
	that looks quite similar to the classical Monge-Kantorovich
	optimal mass 
	transportation problem and in fact, as we argue later,
	is nothing else but the idempotent version of the latter.
	We begin with a short motivational introduction.
	
	\subsection{Motivation of the problem} Suppose
	we have $m$ signal sources
	and $n$ receivers regularly exchanging information between them.  Each 
	source 
	$i\in \{1,\ldots, m\}$ may
	transmit an amount $h_{i,j}$ of information to
	$j\in \{1,\ldots, n\}$.
	The maximum amount of information the source $i$ may send at one time
	is given by a number $k_i$, that is,
	\begin{equation}\label{eq_maxj1}
		\max_{j\in \{1,\ldots, n\}} h_{i,j}= k_i.
	\end{equation}
	Analogously,
	the maximum amount of information the receiver $j$ may get at one time
	is given by a number $l_j$, that is,
	\begin{equation}\label{eq_maxi1}
		\max_{i\in \{1,\ldots, m\}} h_{i,j}= l_j.
	\end{equation}
	Of course,~\eqref{eq_maxj1} and~\eqref{eq_maxi1} may only be
	simultaneously valid if
	\begin{equation}\label{eq_tropmasseq1}
		\max_{i\in \{1,\ldots, n\}} k_i = \max_{j\in \{1,\ldots, n\}} l_j .
	\end{equation}
	The cost $C_{i,j}$ of transmitting
	between the source $i$ and the receiver $j$ depends affinely
	on the amount of transmitted information and takes into account
	the known fixed cost 
	$g_{i,j}$ of using the communication channel between them, that is,
	\[ C_{i,j} = g_{i,j} + \gamma h_{i,j}
	\] 
	for some given coefficient $\gamma>0$.
	The goal is to find the values
	$h_{i,j}$, $i=1,\ldots, n$,  $j=1,\ldots, m$
	(the respective matrix being further called the optimal tropical
	transportation plan, the explanation of the terminology
	being given in the sequel) minimizing the
	maximum of $C_{i,j}$ over all $i$ and $j$, that is,
	finding the
	\[
	\inf \{\max_{i,j}  (g_{i,j} + \gamma h_{i,j}) \colon \mbox{$h_{i,j}$ satisfying~\eqref{eq_maxj1} and~\eqref{eq_maxi1}}  \}.
	\] 
	Denoting $c_{i,j} := g_{i,j}/\gamma$, this amounts to solving
	\begin{equation}\label{eq_pb0maxplus}
		\inf \{\max_{i,j}  (c_{i,j} + h_{i,j}) 
		\colon \mbox{$h_{i,j} $ satisfying~\eqref{eq_maxj1} and~\eqref{eq_maxi1}}\}.
	\end{equation}

	\subsection{Idempotent (max-plus or tropical)
		interpretation}
	Let us now completely
	change the point of view and look at the above problem
	as a version of the classical optimal mass
	transportation problem 
	in the context of 
		\emph{idempotent analysis}:
	more precisely,
	analysis over the tropical (max-plus) semiring
	\(\bar{\R}_{-}:=\R\cup\{-\infty\}\)
	endowed with the operations
	\[
	a\oplus b := \max{\{a,b\}},\quad 
	a \otimes b := a + b,
	\]
	which substitute the usual addition and
	multiplication of real numbers respectively. 
	The value $-\infty$ is an identity with respect
	to $\oplus$ and $0$ is an identity with
	respect to $\otimes$.
	Both operations are commutative,
	associative and $a\otimes (b \oplus c)
	= a\otimes b + a\otimes c$. 
	Thus the roles of $0$ and $1$ on the usual real line 
	are played here by $-\infty$ and $0$ respectively.
	For a general overview of idempotent analysis we refer the
	reader to the classic book~\cite{kolokoltsov1997idempotent}.
	
	The classical discrete Monge-Kantorovich
	optimal mass transportation problem
	(see, e.g.~\cite{Villani06omt} for a comprehensive
	introduction to the subject)
		is that of
		finding the optimal plan of transportation
		in the following sense:
		solve the minimization problem
		\begin{equation}\label{eq_pb0transport}
			\inf\big\{ 
			\sum_{i,j=1}^{m,n}  c_{i,j} \pi_{i,j} \
			\colon
			\ 
			[\pi_{i,j}]_{i,j=1}^{m,n}
			\big\}
		\end{equation}
		where the infimimum is performed over \(m\)-by-\(n\) matrices
		\([\pi_{i,j}]_{i,j=1}^{m,n}\) which satisfy the constraints
		\begin{eqnarray}\label{eq_trj1}
			\sum_{j=1}^n \pi_{i,j}= k_i,\\
			\label{eq_tri1}
			\sum_{i=1}^m \pi_{i,j}= l_j,
		\end{eqnarray}
		with the numbers $k_i$, $l_j$, $i=1,\ldots,m$, $j=1,\ldots,n$
		all fixed.
	This is usually interpreted as finding the way of optimally
	transporting the discrete measure
	\[\mu:=\sum_{i=1}^{m} k_i \delta_{x_i}\] to another  discrete measure
	\[\nu:=\sum_{j=1}^{n} l_j \delta_{y_i},\] for some $x_i\in X$, $y_j\in Y$, $i=1,\ldots,m$, $j=1,\ldots,n$, with $X$ and $Y$ some sets and
	$\delta_z$ standing for the Dirac point mass at \(z\).

		The value $\pi_{i,j}$ is, then, interpreted as the amount of
		mass transported from $x_i$ to $y_j$. 
		The matrix  \([\pi_{i,j}]_{i,j=1}^{m,n}\)
		is identified with the discrete measure
		\(\pi= \sum_{i,j=1}^{m,n} \pi_{i,j}\delta_{(x_i, y_j)}\)
		over $X\times Y$;
		constraints \eqref{eq_trj1} and \eqref{eq_tri1} now mean
		that the marginals (or projections)
		of the measure \(\pi\) along \(X\) and \(Y\)
		are \(\mu\) and \(\nu\) respectively.
		The quantity
		\(\sum_{i,j=1}^{m,n}  c_{i,j} \pi_{i,j}\)
		is the total transportation cost, targeted for minimization.
	
	In the idempotent max-plus setting the
	role of the Dirac measure $\delta_z$
		over an arbitrary set \(Z\) 
		concentrated at a point $z\in Z$
		is played by the characteristic function
		(for which we retain the same notation as for the Dirac measure)
		\(\delta_z\) defined by
		\[
		\delta_z (z') := 
		\begin{cases}
			0, & z' = z,
			\\
			-\infty, & z'\neq z.
		\end{cases}
		\]
		The analogues of sums of Dirac masses on 
		sets \(X\) and \(Y\) are the functions
		on these sets
		respectively 
		defined by  
		\begin{equation}\label{eq:formmunu}
			\mu(x) := \max\limits_{i=1,\ldots, m}(k_i + \delta_{x_i}(x)), \qquad 
			\nu(y) = \max\limits_{j=1,\ldots, n} (l_j + \delta_{y_j}(y)),
	\end{equation} 
	i.e.\ $\mu$ is the function taking the value $k_i$
	at each $x_i$ and $-\infty$ elsewhere,
	and $\nu$ is the function taking the value $l_j$ at
	each $y_i$ and $-\infty$ elsewhere; 
		the analogue of a discrete measure represented by a sum of Dirac masses with 
		weights \(h_{ij}\) at points 
		\((x_i,y_j)\in X\times Y\)  is the function
		\begin{equation}\label{eq:forpi}
			\pi(x,y) = \max_{\substack{i=1,\ldots,m \\ j=1,\ldots,n}}
			\left( h_{ij} + \delta_{(x_i,y_j)}(x,y) \right).
	\end{equation}
	We will be referring to the coefficients $k_i$ as
	the \textit{weights} of $\mu$ and to the coefficients $l_j$
	as the weights of $\nu$.
		The total mass of a discrete measure,
		which in the traditional setting is the sum of its weights,
		corresponds, in the max-plus setting, to the maximum of its weights, i.e.\
	\[
	|\mu| := \max\limits_{i=1,\ldots, m}k_i , \qquad 
	|\nu|: = \max\limits_{j=1,\ldots, n} l_j .
	\] 
	We will assume, in complete analogy with the
	classical mass transportation theory,
	that 
		\(|\mu|=|\nu|\),
	which is exactly the condition~\eqref{eq_tropmasseq1},
	and for purely aesthetical reasons,
	which imply no loss of generality, 
	we also assume that both total masses are zero, 
	i.~e.~$|\mu|=|\nu|=0$, so that $\mu$ and $\nu$ can
	be considered tropical versions of discrete
	\textit{probability} measures.
		We will call, therefore, functions such as
		\(\mu\), \(\nu\), \(\pi\) above 
		\textit{discrete max-plus probability measures}, the set of such
		functions over a
		given set $Z$ being denoted $\mathcal{M}(Z)$, so that
		$\mu\in \mathcal{M}(X)$ and $\nu\in \mathcal{M}(Y)$,
		\(\pi\in \mathcal{M}(X\times Y)\).

		Suppose now that 
		\(\{x_i\}_{i=1}^m\subset  X\), 
		\(\{y_j\}_{j=1}^n\subset  Y\)
		are given, and 
		\(\mu\), \(\nu\) are defined by~\eqref{eq:formmunu} and
		\eqref{eq:forpi} respectively. The max-plus, or tropical, analogue
		of a transport plan between \(\mu\) and \(\nu\) is a function
		\(\pi\) defined as in \eqref{eq:forpi} and satisfying the constraints
		\begin{equation}\label{eq_planpi1}
			\max_{x\in X}\pi(x,y) = \nu(y), \quad \max_{y\in Y}\pi(x,y) = \nu(x),
		\end{equation}	
		The Monge-Kantorovich
		optimal transportation problem~\eqref{eq_pb0transport} with the given cost function 
		$c\colon X\times Y \to \bar \R$ 
		then becomes, in the max-plus setting, the problem of
		solving
		\begin{equation}\label{eq:xyzw}
			\inf\left\{\max(  c(x,y) + \pi(x,y) )\colon \pi\in
			\mathcal{M}(X\times Y)\, \mbox{ satisfies~\eqref{eq_planpi1}}\right\}.
		\end{equation}

		It is worth mentioning that the problem just stated is not the 
		unique example of a meaningful idempotent
		(max-plus or tropical) version of 
		a classical optimization problem; 
		similar tropical formulations have arisen elsewhere
		in the literature. For instance, this
		is the case of the so-called bottleneck traveling salesman problem
		(see e.g. section~8 of~\cite{GimGomor64-tsp}
		or~\cite{GarfGilb78-bottlenecktsp}), which can be
		considered a max-plus version of the classical traveling salesman
		problem. 
	
		We will further identify, whenever convenient, max-plus
		discrete probability measures with the
		sequences of their weights,
		and the transport plan $\pi$ (given by~\eqref{eq:forpi})
		with 
		the matrix of coefficients \([h_{i,j}]\),
		and refer to this object in either interpretation as a
		\emph{tropical transport plan} for the discrete
		max-plus probability measures $\mu$ and $\nu$
		(or, equivalently, for the sequences of their weights)
		whenever~\eqref{eq_planpi1} holds,
		which in terms of the matrix \([h_{i,j}]\) amounts
		precisely to~\eqref{eq_maxj1} and~\eqref{eq_maxi1}, namely,
		that maximum of the \(i\)-th row of the matrix must be \(k_i\)
		and the maximum of the \(j\)-th row of the matrix must be \(l_j\).
		If we write \(c_{i,j}:=c(x_i,y_j)\), then,
		in view of \eqref{eq:formmunu} and \eqref{eq:forpi}, the problem
		\eqref{eq:xyzw}
		becomes exactly~\eqref{eq_pb0maxplus}, 
		which is the reason why it may be considered 
		as the max-plus version of the
		Monge-Kantorovich problem~\eqref{eq_pb0transport}.
		Such an identification of measures with weights,
		plans and cost functions with matrices is quite natural
		in the discrete setting we are considering here,
		especially when the points $x_i$ and $y_j$ themselves are of
		no practical importance.
	
	\subsection{Our contribution}
	In this paper we provide an
	explicit algorithm to solve the optimal tropical transportation
	problem~\eqref{eq_pb0maxplus} and find an explicit formula
	for the optimal tropical cost, i.~e.~the value of \eqref{eq_pb0maxplus}.
	As a consequence, we obtain some
	curious results on the optimal
	tropical plans and values. In particular,
	\begin{itemize}
		\item In the case $m=n$ optimal tropical
		plans corresponding to perfect matchings
		(those given by permutation matrices)
		may
		not exist even if the max-plus probability measures $\mu$
		and $\nu$ have all the weights equal to zero
		(we henceforth call this case \textit{fundamental}), see Example \ref{example:ultrametricnotenough} below.
		This is a stark contrast with classical optimal
			mass transportation theory, 
			where (again with $m=n$) perfect matching optimal transport
			plans between sums of Dirac masses with equal
			weights always exist. Nevertheless,
			it turns out
			that, at least in the fundamental case,
			the existence of perfect
			matching optimal tropical plans occurs rather
			frequently as the number of weights of
			both $\mu$ and $\nu$ becomes large, the respective statement being made precise by introducing randomness in the cost.
		More precisely, under a concrete randomization of the cost matrix,
			the existence of a perfect matching optimal plan is ``asymptotically
			almost sure'' as the number of weights of the measures approaches
			infinity. This is Theorem \ref{th:pmtendstoone} below.
			\item In the fundamental case,
			under the same type of randomization of the cost matrix,
			the optimal tropical cost is, asymptotically almost surely,
			the lowest value among all the entries of the cost matrix.
			This is the content of Theorem \ref{thm:bernoulli1} and
			Remark \ref{rem-binomial} below.
			\item 
			We also prove that uniqueness of an optimal
			tropical plan asymptotically almost surely fails to occur
			(in the fundamental case), when the cost matrix entries are
			sampled uniformly. This is Theorem \ref{thm:rarity1} below.
	\end{itemize}
	
	\section{Notation and preliminaries}
	
	In complete analogy with the classical optimal transportation theory,
	the matrix \([h_{i,j}[_{i,j=1}^{m,n}\) with each
	$h_{i,j}\in [-\infty,0]$ satisfying~\eqref{eq_maxj1}
	and~\eqref{eq_maxi1} will be called discrete
	\textit{max-plus (or tropical)  
		plan} (or just a \textit{plan} for brevity) for
	max-plus discrete probability measures
	$\mu\in\mathcal{M}(X),$ $\nu\in\mathcal{M}(Y)$.
		Equivalently,
		as remarked earlier,
		it can be seen as a
		max-plus discrete probability measure
		in the sense given by \eqref{eq:forpi}.
	We denote by $\Pi(\mu,\nu)$  the set of all such plans
	(which is always
	nonempty, since
	$\mu\otimes \nu\in \Pi(\mu,\nu)$, where $(\mu\otimes \nu)_{i,j}:=k_i+l_j$).
	
	For the given cost matrix \([c_{i,j}]_{i,j=1}^{m,n}\)
	we define
	\begin{align*}
		d_c(\mu,\nu) := \ \inf\left\{ \max_{\substack{i=1,\ldots, m\\ j=1,\ldots, n}}(c_{i,j}+h_{i,j})
		\colon h \in \Pi(\mu,\nu) \right\}.
	\end{align*}
	If we interpret $h$ as an element of 
	$h\in \mathcal{M}(X\times Y)$, i.e.\
	as in \eqref{eq:forpi},
	then we may write
	$h(x_i,y_j)$ and $c(x_i,y_j)$ instead of $h_{i,j}$
	and $c_{i,j}$ respectively, since the points \(x_i\) 
	and \(y_j\) can be assumed fixed in every discussion.
	Again for purely aesthetical reasons, and to allow for
	the interpretation of the numbers $c_{i,j}$ as representing a cost,
	it is convenient to assume
	$c_{i,j}\geq 0$, which can always be done without
	loss of generality. 
	The minimizer  \(h\in\Pi(\mu,\nu)\) in the above
	problem will be called the
	\textit{minimizing (or optimal) tropical plan},
	the set of such minimizing plans being
	denoted by $\Pi^c(\mu,\nu)$. 
		The number \(d_c(\mu,\nu)\) will be called the
		\textit{optimal tropical cost} between
		\(\mu\) and \(\nu\). We must say that, despite our choice
		of notation, the function \(d_c(\cdot,\cdot)\) is not a
		metric.
	
	In the sequel we assume the sequences of weights \(k_j\)
	and \(l_j\) to be ordered in decreasing order
		$k_1=l_1=0$, i.e.\
	\begin{equation}\label{eq:standingassumption}
		k_{n} \leq k_{n-1} \leq \cdots \leq k_1 = 0, 
		\quad 
		l_{n} \leq l_{n-1} \leq \cdots \leq l_1 = 0.   
	\end{equation}
	We denote by 
	\(\Lambda(\mu)\) and \(\Lambda(\nu)\) the 
	sets of weights of \(\mu\) and \(\nu\)
	respectively. If we wish to retain the
		interpretation of \(\mu\) and \(\nu\) as elements of
		\(\mathcal{M}(X)\) and \(\mathcal{M}(Y)\) respectively, then
		\eqref{eq:standingassumption} is achieved simply by a relabeling
		of the fixed points \(x_i\) and \(y_j\), \(i\in\{1,\ldots,m\}\),
		\(j\in \{1,\ldots,n\}\).
	
	For any \(h\in \Pi(\mu,\nu)\), by the
	\textit{support} of \(h\), denoted \(\supp(h)\), we will mean
	the subset of \(X\times Y\) of 
	points \((x,y)\) where \(h(x,y) > -\infty\), or (again, equivalently,
	since any such point must be one of the pairs \((x_i,y_j)\))
	with the set of pairs 
	\((i,j)\in \{1,\ldots,m\}\times \{1,\ldots,n\} \)
	such that \(h_{i,j} > -\infty\).
		In the latter case, we may also write \(h(i,j)\) rather
		than \(h_{ij}\) (for instance, if we wish to free up the subindex
		place for another purpose, as in section
		\ref{subsection:findingoptimalcost} below).
	
	For a set $X$ we denote by $\# X$ its
	cardinality.
	We also write sometimes $a\vee b$ for the maximum of the numbers
	$a$ and $b$.
	
	\section{Reduced transportation plans and existence of minimizers}
	
	We start with the following definition.
	\begin{defn}\label{def_reducedplan1}
		Given fixed discrete max-plus probability measures
		\(\mu\) and \(\nu\),
		we will call a tropical plan \(h\in \Pi(\mu,\nu)\) 
		\emph{reduced} if 
		for each \(i,j\) such that \(h_{i,j}>-\infty\), 
		the element \(h_{i,j}\) is a \emph{strict} 
		maximum in its row \emph{or} in its column, and denote
		by \(\Pi_{R}(\mu,\nu)\) the set of reduced plans
		for discrete \(\mu\) and \(\nu\).
	\end{defn}	
	Without loss of
	generality for the optimal tropical transportation problem,
	all the weights of \(\mu\) and
	\(\nu\) can be taken to be finite (i.~e.~\(>-\infty\)). 
	In fact, if, say, \(k_i=-\infty\) for some \(i\in\{1,\ldots,m\}\), 
	then the \(i\)-th row of \({h}\), 
	for any \({h}\in \Pi(\mu,\nu)\) 
	must consist only of \(-\infty\). In this case,
	in the expression that defines 
	\(d_c(\mu,\nu)\),
	each of the elements over which the minimum is taken is
	\begin{align*}
		& \max_{(i,j)} ({h}_{i,j}+ c_{i,j}) 
		\\
		& \ =  
		\max\{ \ldots, h_{i,1} + c_{i,1}, h_{i,2} + c_{i,2}, \ldots ,
		h_{i,n} + c_{i,n}, \ldots \}
		\\
		& \  = \max\{ \ldots, -\infty, -\infty, \ldots
		-\infty, \ldots \},
	\end{align*}
	but the maximum is non-negative, so the
	the numbers \(-\infty\) can 
	be changed to sufficiently small negative numbers 
	(negative but with large absolute value) without affecting
	the maximum and then the weight \(k_j=-\infty\)
	can be changed to \(\max_i h_{i,j}\) 
	where \(h_{i,j}\) are the new numbers just mentioned.
	
	The following assertion holds true.
	
	\begin{lemma}\label{lemma:red}
		For all discrete \(\mu\in \mathcal{M}(X)\), 
		\(\nu\in\mathcal{M}(Y)\) one has
		\begin{align*}
			d_c(\mu,\nu) = 
			\inf\{ \max_{(i,j)}(h_{i,j} + c_{i,j})
			\colon  \ h\in \Pi_R(\mu,\nu)\}.
		\end{align*}
		Moreover, for every minimizing plan $h$
		there is a reduced minimizing plan $\tilde h$ with
		$\supp\, \tilde h \subset \supp\, h$ and $\tilde h =h$
		on the support of $\tilde h$. 
	\end{lemma}
	
	\begin{proof}
		If
		\(h_{i,j}\) is not a strict maximum neither in its
		column nor in its row for some \(i,j \in \{1,\ldots,n\}\), then
		changing  \(h_{i,j}\) to $-\infty$  (or to any number less than 
		\(h_{i,j}\))  can only decrease
		\(\max_{(i,j)}(h_{i,j} + c_{i,j})\).
		Changing all such entries of the matrix 
		\([h_{i,j}]\) will transform the plan to
		a reduced one, and thus
		\begin{align*}
			d_c(\mu,\nu) & = 
			\inf\{ \max_{(i,j)}(h_{i,j} + c_{i,j})
			\colon  \ h\in \Pi(\mu,\nu)\}\\
			& =	
			\inf\{ \max_{(i,j)}(h_{i,j} + c_{i,j})
			\colon  \ h\in \Pi_R(\mu,\nu)\}
		\end{align*}
		as claimed.
	\end{proof}
	As a consequence, the following existence result holds.
	\begin{thm}\label{thm:infismin} 
		The discrete 
		max-plus transportation problem admits a solution, namely,
		\(\inf\) is actually a \(\min\).
	\end{thm}
	
	\begin{proof}
		It is enough to refer to Lemma~\ref{lemma:red} and
		observe that the set of reduced plans 
		\(\Pi_R(\mu,\nu)\) has finitely many elements
			(indeed, each entry of a reduced plan must be
			either \(-\infty\) or one of the weights of
			\(\mu\) and \(\nu\)).
	\end{proof}
	
	\section{Algorithm to solve the discrete max-plus transportation problem}
	
	\subsection{Partition of the support of a plan}
	
	Given discrete \(\mu\) and \(\nu\), for each 
	\(i\in \{1,\ldots,m\}\), \(j\in \{1,\ldots,n\}\), let
	\begin{align*}
		p_i = \max\{j \colon  \ l_j \geq k_i\}, 
		\quad 
		q_j = \max\{i \colon  \ k_i \geq l_j\}, 
		\\
		S_i = \{(i,1),\ldots,(i,p_i)\},
		\quad
		T_j = \{(1,j),\ldots,(q_j,j)\}.  
	\end{align*}
	The following statement gives some information
		on the general structure of reduced plans, as long as we
		adhere to the convention that the weights are sorted as in
		\eqref{eq:standingassumption}, which we agreed to hold throughout.
	\begin{lemma}\label{lemma:positions1}
		Let \(\mu\in\mathcal{M}(X)\),
		\(\nu\in\mathcal{M}(Y)\) be
		discrete max-plus probability measures as
		in \eqref{eq:formmunu} and let
		\(h \in \Pi_R(\mu,\nu)\).
			Assume, without loss of generality, that the weights
			of the measures satisfy \eqref{eq:standingassumption}.
		The following assertions hold true.
		\begin{enumerate}
			\item For each \(i\in\{1,\ldots,m\}\), 
			at least one of the numbers \(h_{i,1},\ldots,h_{i,p_i}\)
			must be \(k_i\), and the numbers \(h_{i,p_i+1},\ldots,h_{i,n}\)
			are all strictly less than \(k_i\). Likewise,
			for each \(j\in\{1,\ldots,n\}\), at least one 
			of the numbers \(h_{1,j},\ldots,h_{q_{i},j}\) must be \(l_j\),
			and the numbers \(h_{q_{i+1},j},\ldots,q_{n,j}\) are all
			strictly less than \(l_j\).
			\item If the weights \(k_i\) and \(l_j\) are all distinct, with 
			the exception of \(k_1 = l_1 = 0\), then 
			\(S_i\cap T_j = \emptyset\) whenever \((i,j)\neq (1,1)\).
			\item One has \(k_i = l_j\) for some \((i,j)\in\{1,\ldots,m\}\times \{1,\ldots,n\}\),
			if and only if \((i,j)\in S_i\cap T_j\).
		\end{enumerate} 
	\end{lemma}
	\begin{proof}
	(1) Fix \(i\in \{1,\ldots,m\}\). 
			The maximum among
			\(h_{i,1},\ldots,h_{i,n}\) must be \(k_i\). If
			\(h_{i,p_i+\bar m}= k_i\) for some \(\bar m> 0\),
			then the maximum among \(h_{1,p_i+\bar m},\ldots,h_{n,p_i+\bar m}\)
			is at least \(k_i\). 
			The maximum among 
			\(h_{1,p_i+m},\ldots,h_{n,p_i+\bar m}\) must be 
			\(l_{p_i+\bar m}\), which, by definition of $p_i$ is,
			strictly less than $k_i$. 
			This contradiction proves that 
			the maximum of \(h_{i,1},\ldots,h_{i,n}\), equal to $k_i$, 
			occurs among
			\(h_{i,1},\ldots,h_{i,p_i}\), and not among 
			\(h_{i,p_i+1}, \ldots, h_{n}\),
			which proves the first part
			of the assertion. The second part,
			i.e.~the claim about the numbers
			\(h_{1,j},\ldots,h_{n,j}\) is proven completely symmetrically. 
		
		(2) Suppose \((i,j)\neq (1,1)\) and 
		\((q,p)\in S_i\cap T_j\).
			Since the pair \((q,p)\) is in \(S_i\),
			its first component must be \(i\), i.e.\ \(q=i\).
			Similarly, since it is in \(T_j\), we must have
			\(p=j\).
			Thus \((q,p) = (i,j)\).
			Moreover, the definition
			of \(p_i\) and \(q_j\) now contains only strict
			inequalities because we are assuming all the weights
			distinct with the exception of \(k_1=l_1=0\).
			Having \((i,j)\in S_i\), then, implies that
			\(l_j > k_i\), while having \((i,j)\in T_j\)
			implies that \(k_i > l_j\), and we have obtained
			a contradiction.
		
		(3) Suppose \(k_i = l_j\) for some pair
		\((i,j)\in\{1,\ldots,m\}\times \{1,\ldots,n\}\).
		Since
		\(l_j \geq k_i\), we must have \((i,j)\in S_i\). Likewise,
		since \(k_i\geq l_j\), then \((i,j)\in T_j\), so that  necessity is proven.
		Now suppose \((i,j)\in S_i\cap T_j\). Since \((i,j)\in S_i\),
		\(j\leq p_i\) so \(l_j \geq k_i\), and \((i,j)\in T_j\) gives
		\(i\in T_i\), so \(k_i\geq l_j\). This completes the proof.		
	\end{proof}
	
	Given discrete
	max-plus probability measures $\mu,\nu$ and a real number \(\lambda\),
	let 
	\begin{equation}\label{eq:defofRlambda}
		R_{\lambda} :=
		\bigg(\bigcup\limits_{\{i \colon  \ k_i=\lambda\}} S_i\bigg) 
		\cup 
		\bigg(\bigcup\limits_{\{j \colon  \ l_j=\lambda\}} T_j\bigg),		
	\end{equation}
	which is a subset of \(\{1,\ldots,m\}\times \{1,\ldots,n\}\).
	We call \(R_{\lambda}\) a 
	\textit{region} or \(\lambda\)-region to emphasize
	the dependence on \(\lambda\). A region can look like
	an L written backwards (like the one in pink in Figure
	\ref{fig:regions} below),
	with the ends resting on the top and left edges of the grid,
	or a rectangle with its left side lying on
	the left edge of the grid, or a rectangle 
	with its top side on the top edge of the grid,
	or a rectangle with both its left and top sides
	lying on the left and top sides of the grid,
	respectively. 
		We remark that that our notion of region exists only when
		the measures \(\mu\) and \(\nu\) have been fixed. Also, 
		for the description of our algorithm, it is essential that the
		weights of these measures are labeled as in \ref{eq:standingassumption}.
	\begin{example}\label{example:regions}
		For $m=n=6$ and the max-plus probability measures 
		\begin{align*}
			\mu & = \max\{ 0 +\delta_{x_1},
			0+ \delta_{x_2}, -2 + \delta_{x_3},
			-3 + \delta_{x_4}, -4 + \delta_{x_5},
			-4 + \delta_{x_6}\}\\
			& = \left\{
			\begin{array}{rl}
				0, & x\in \{x_1.x_2\},\\
				-2, & x=x_3,\\
				-3, & x=x_4,\\
				-4, & x\in \{x_5.x_6\}
			\end{array}
			\right.
		\end{align*}
		and
		\begin{align*}
			\nu & = \max\{ 0 +\delta_{y_1},
			0+ \delta_{y_2}, 0 + \delta_{y_3},
			-1 + \delta_{y_4}, -2 + \delta_{y_5},
			-2 + \delta_{y_6}\}\\
			& = \left\{
			\begin{array}{rl}
				0, & y\in \{y_1.y_2.y_3\},\\
				-1, & y=y_4,\\
				-2. & y\in \{y_5.y_6\}
			\end{array}
			\right.
		\end{align*}		
		with  $x_j$, $j=1,\ldots, m$ as well as $y_i$, $i=1,\ldots, n$  all distinct,
		the regions (each in a different color)
		and a plan
		are shown in Figure \ref{fig:regions}.
		\begin{figure}[!htbp]
			\begin{center}
				\begin{tikzpicture}
					\matrix [nodes={text width=6mm, text height=3mm, align = center}]
					{
						\node(a){};		& \node {0}; & \node {0}; & \node {0};  & \node {-1};& \node {-2}; & \node {-2};\\
						\hline
						\node {0};		& \node[fill=cyan!40]{-$\infty$}; & \node[fill=cyan!40]{-$\infty$}; & \node[fill=cyan!40] {0}; & \node[fill=yellow] {-1}; & \node[fill=magenta!40]{-$\infty$}; & \node[fill=magenta!40] {-2};\\
						\node {0};		&\node[fill=cyan!40]{0}; & \node[fill=cyan!40]{0}; & \node[fill=cyan!40] {-$\infty$}; & \node[fill=yellow] {-$\infty$}; & \node[fill=magenta!40]{-2}; & \node[fill=magenta!40] {-$\infty$};\\
						\node {-2};		& \node[fill=magenta!40]{-$\infty$}; & \node[fill=magenta!40]{-$\infty$}; & \node[fill=magenta!40] {-2}; & \node[fill=magenta!40] {-$\infty$}; & \node[fill=magenta!40]{-$\infty$}; & \node[fill=magenta!40] {-$\infty$};\\
						\node {-3};	&\node[fill=green]{-$\infty$}; & \node[fill=green]{-$\infty$}; & \node[fill=green]{-$\infty$}; & \node[fill=green] {-$\infty$}; & \node[fill=green]{-3}; & \node[fill=green] {-$\infty$};\\
						\node {-4};	& \node[fill=blue!40]{-$\infty$}; & \node[fill=blue!40]{-$\infty$}; & \node[fill=blue!40]{-$\infty$}; & \node[fill=blue!40] {-4}; & \node[fill=blue!40]{-$\infty$}; & \node[fill=blue!40] {-$\infty$};\\
						\node(b) {-4};	& \node[fill=blue!40]{-4}; & \node[fill=blue!40]{-$\infty$}; & \node[fill=blue!40]{-$\infty$}; & \node[fill=blue!40] {-$\infty$}; & \node[fill=blue!40]{-$\infty$}; & \node[fill=blue!40] {-$\infty$};\\
					};
					\draw [thick]  (a.north east) -- (b.south east);
				\end{tikzpicture}
				\caption{Regions for the pair 
					\((\mu,\nu)\) of Example \ref{example:regions}.}
				\label{fig:regions}
			\end{center}
		\end{figure} 
	\end{example}
	It is convenient to extend the notions of plan and
	reduced plan as follows. 
	Fix
	discrete max-plus probability measures $\mu, \nu$,
	with their weights arranged as in \eqref{eq:standingassumption};
	suppose
	\(\lambda\) is one of these weights and consider the corresponding
	region \(R_{\lambda}\). 
		By a 
		\textit{plan} of $R_{\lambda}$
		we will mean a function \(h: R_{\lambda} \to [-\infty,0]\)
		such that the maximum of \(h\) on 
		each row and on each column of \(R_{\lambda}\)
		is $\lambda$. In Figure \ref{fig:regions}
		we see plans of each of the five regions, determined by the numbers
		in the cells.

	Let $\Pi(R_{\lambda})$ 
	be the set of plans of $R_{\lambda}$. Like above, a plan 
	$h=\{h_{i,j}\}_{(i,j)\in R_{\lambda}} \in \Pi(R_{\lambda})$ is 
	called \textit{reduced} whenever $h_{i,j}$ is a strict maximum 
	of its row or a strict maximum of its column,
	as long as $h_{i,j}> -\infty$. Thus, a reduced plan
	of a \(\lambda\)-region has no numbers other
	than \(-\infty\) and \(\lambda\). The plans of the regions
	in Figure \ref{fig:regions}
	are all reduced.
	We will denote by $\Pi_R(R_{\lambda})$  the set 
	of reduced plans of $R_{\lambda}$.
	
		Given discrete max-plus probability measures $\mu, \nu$,
		a region \(R_{\lambda}\), and
		a cost function $c$, we will use the notation \(d_c\) to
		also mean the following:
	\begin{equation*}
		d_c(R_{\lambda})
		:= 
		\min\limits_{h\in\Pi(R_{\lambda})}
		\max\limits_{(i,j)\in R_{\lambda}}
		(h_{i,j} + c_{i,j}).
	\end{equation*} 
		A plan \(h\in \Pi(R_{\lambda})\) at which the \(\min\)
		in the preceding formula is attained will be called
		a \textit{minimizing (or optimal) plan for the region} \(R_{\lambda}\).
	The following assertion holds true.
	\begin{prop}\label{prop:decompose}
		Let
		\(\mu\in \mathcal{M}(X), \nu \in \mathcal{M}(Y)\)
		be arbitrary discrete max-plus probability measures and a 
		cost function \(c:X\times Y\to [0,\infty)\) be given.
		Then
		\[
		d_c(\mu,\nu)=\max_{\lambda\in \Lambda (\mu)\cup
			\Lambda(\nu)}d_c(R_{\lambda}).
		\]
	\end{prop}
	\begin{proof}
		By definition, 
		\[
		d_c(\mu,\nu) = 
		\min_{h\in \Pi(\mu,\nu)}
		\max_{(i,j)}
		(h_{i,j}+c(x_i,y_j)).
		\]
		Let us look at 
		\[
		M =
		\max_{\lambda}
		\min_{h\in\Pi(R_{\lambda})}
		\max_{(i,j)\in R_{\lambda}}
		(h_{i,j}+c(x_i,y_j)),
		\]
		which is the right hand side of the
		inequality we wish to prove. For each one of the distinct 
		\(\lambda\)'s, we pick \(h^{\lambda}\in R_{\lambda}\)
		for which \(\max_{(i,j)\in R_{\lambda}}(h_{i,j}+c(x_i,y_j))\)
		takes the least possible value, i.e.~we pick an optimal
		plan \(h^{\lambda}\)
		of the region \(R_{\lambda}\) for each \(\lambda\).
		Further, let \(\bar\lambda\) be the value of \(\lambda\) 
		at which \(M\) is attained. 
			Let \(h^{\ast}\) be the element of \(\Pi(\mu,\nu)\)
			such that its restriction to \(R_{\lambda}\) is
			\(h^{\lambda}\), for each \(\lambda\in \Lambda(\mu)\cup
			\Lambda(\nu)\).
		We claim that \(h^{\ast}\) is 
		optimal for \(d_c(\mu,\nu)\). In fact, if it is not,
		then  there is another \(h^0\in \Pi(\mu,\nu)\) such that
		\[
		\max_{(i,j)}
		(h^0_{i,j}+c(x_i,y_j))
		\leq \max_{(i,j)}
		(h_{i,j}+c(x_i,y_j)) \quad \forall h\in \Pi(\mu,\nu).
		\] 
		In particular, if \(h=h^{\ast}\), then, by the assumption just
		made, the inequality must be strict, and
		\begin{align*}
			\max_{(i,j)\in R_{h^{\ast}}}
			(h^0_{i,j}+c(x_i,y_j))
			\leq \max_{(i,j)}
			(h_{i,j}^0+c(x_i,y_j))
			< \max_{(i,j)} (h^{\ast}_{i,j}+c(x_i,y_j)).
		\end{align*}
		But 
		the maximum value of the function
		\(\lambda \mapsto \max_{(i,j)\in R_{\lambda}}
		(h^{\ast}_{i,j}+c(x_i,y_j))\)  
		is \(M\) and is attained at \(\lambda = h^{\ast}\).
		Thus, it follows
		that
		\begin{align*}
			\max_{(i,j)\in R_{\lambda^{\star}}}
			(h^0_{i,j}+c(x_i,y_j))
			< 
			\max_{(i,j)\in R_{\lambda^{\star}}}
			(\lambda^{\star}_{i,j}+c(x_i,y_j)) = 
			\max_{(i,j)\in R_{\lambda^{\star}}}
			(h^{\bar{\lambda}}_{i,j}+c(x_i,y_j)),
		\end{align*}
		which contradicts the definition of \(h^{\bar\lambda}\).
		Therefore, \(\lambda^{\star}\) is optimal for \(d_c(\mu,\nu)\),
		so \(d_c(\mu,\nu)=M\).
	\end{proof}
	
	\subsection{Finding the optimal cost on a region}
	\label{subsection:findingoptimalcost}
	By Proposition~\ref{prop:decompose}, to solve the original
	problem, it is enough to find the optimal
	plan for each $\lambda$-region $R_\lambda$,
	hence also finding the respective
	optimal costs $d_c(R_{\lambda})$; the optimal
	plan for the original problem will then coincide 
	over each $R_\lambda$ with the optimal plan for
	this region.
	
		To find the optimal
		plan for the given region $R_\lambda$, suppose the 
		cost function $c$ be given;
		let us number the values that $c$ takes over $R_\lambda$ in
		an increasing order. Namely, suppose that \(s\in \mathbb{Z}^+\) be
	the number of distinct values that \(c\) takes on over the region
	\(R_{\lambda}\) and 
	denote these values, in increasing order, by
	\begin{equation}\label{eq:c1cs}
		\beta_1 < \cdots < \beta_s.
	\end{equation}
	For each $m \in \{1,2, \ldots, s\}$ 
	we define the function
	$h_{c}^m\colon R_\lambda\to \{-\infty,\lambda\} $
	by the formula
	\[	
	h_{c}^m(i,j):=
	\left\{
	\begin{array}{rl}
		\lambda, &  \mbox{if $c(x_i,y_j)\leq \beta_m$},\\
		-\infty, &  \mbox{otherwise}.
	\end{array}
	\right.
	\] 
	That is,
		\(h_c^m\) 
		is a plan for the region \(R_{\lambda}\) 
		such that  \(\lambda\) appears in the cells
		that host one of the smallest \(m\) values of \(c\) on the region, while 
		\(-\infty\) appears in all the other cells.
	In particular,
	for $m=s$, \(h_c^s\) fills 
	all the cells in the region \(R_{\lambda}\) with \(\lambda\), 
	and hence is a plan for \(R_{\lambda}\), that is,
	\(h_c^s \in \Pi(R_{\lambda})\).
	This motivates the following definition.
	
	\begin{defn}\label{defn:mclambda}
		Given \(\lambda\), a \(\lambda\)-region \(R_{\lambda}\),
		and a cost function 
		$c$, let \(m\) be the smallest integer for
		which the function \(h_c^m\) on 
		the region \(R_{\lambda}\) constitutes a plan for \(R_{\lambda}\),
		i.~e.
		\begin{align*}
			m_c(\lambda)=\min\{m \colon  \ h_c^m \in \Pi(R_{\lambda})\}.
		\end{align*}
	\end{defn}
	It is convenient to assign
	to each \((i,j)\in R_\lambda\) the number (from \(1\) to \(s\))
	that the value \(c(x_i,y_j)\) occupies in the list
	\eqref{eq:c1cs}. Such an assignment is given by the function
	$f: R_{\lambda} \to \{1,2, \ldots, s\}$ 
	determined by the condition:
	{\small
		\begin{equation}\label{eq_def_fnum1}
			f(i_1,j_1) < f(i_2,j_2) 
			\quad \textrm{if and only if} \quad c(x_{i_1},y_{j_1}) 
			< c(x_{i_2},y_{j_2}) \quad \mbox{for $(i_1,j_1),(i_2,j_2)
				\in R_{\lambda}$.}
		\end{equation}
	}
	We
	illustrate the above definitions with the following example.
	\begin{example}\label{example:c1cs}
		Suppose the region is \(\{1,2,3\}^2\) and the cost 
		function (restricted to this region) is, in matrix form,
		\[
		[c(x_i,y_j)]_{i,j=1}^3 = 
		\begin{pmatrix}
			2 & 4 & 8
			\\
			8 & 2 & 0
			\\
			2 & 0 & 5
		\end{pmatrix}.
		\]
		Then \(f(2,3) = f(3,2) =  1\),
		\(f(1,1)=f(2,2)=f(3,1)=2\),
		\(f(1,2)=3\),  \(f(1,3)=f(2,1)=4\), and
		\begin{align*}
			&\ h_c^1 = 
			\begin{pmatrix}
				-\infty & -\infty & -\infty
				\\
				-\infty & -\infty & \lambda
				\\
				-\infty & \lambda & -\infty
			\end{pmatrix}, \quad
			h_c^2 = 
			\begin{pmatrix}
				\lambda & -\infty & -\infty
				\\
				-\infty & \lambda & \lambda
				\\
				\lambda & \lambda & -\infty
			\end{pmatrix},
			\\
			&\ h_c^3 = 
			\begin{pmatrix}
				\lambda & -\infty & -\infty
				\\
				-\infty & \lambda & \lambda
				\\
				\lambda & \lambda & \lambda
			\end{pmatrix}, \quad
			h_c^4 = 
			\begin{pmatrix}
				\lambda & \lambda & \lambda
				\\
				\lambda & \lambda & \lambda
				\\
				\lambda & \lambda & \lambda
			\end{pmatrix}.
		\end{align*}
		Here \(m_c(\lambda) = 2\),
		\(h_c^{m_c(\lambda)}= h_c^2\).
	\end{example}

	\begin{lemma}\label{lemma:supportsub}
		Let \(R_{\lambda}\) be a \(\lambda\)-region, \(c\)
		be a cost function. Let \(h\in \Pi(R_{\lambda})\)
		be a minimizer for \(d_c(R_{\lambda})\). Then the support 
		of \(h\) is included in the support of \(h_c^{m_c(\lambda)}\) and,
		with the notation of \eqref{eq:c1cs}, 
		\[
		d_c(R_{\lambda}) = \lambda + \beta_{m_c(\lambda)}.
		\] 
		Moreover, \(h_c^{m_c(\lambda)}\) is itself a minimizing plan.
	\end{lemma}
	
	\begin{proof}
		Let \(\{(x_{i_1},y_{j_1}),\ldots,(x_{i_p},y_{j_p})\}\) be the support 
		of \(h\). Then 
		\[
		d_c(R_{\lambda}) = \max_{1\leq k\leq p}
		\{c(x_{i_k},y_{j_k})+\lambda\}.
		\]
		With the notation of \eqref{eq:c1cs}, 
		let \(\beta_m\) be the largest of the \(c(x_{i_k},y_{j_k})\); 
		then \(d_c(R_{\lambda}) = \lambda + \beta_m\). But then the function
		\(h_c^m\), by definition, must place a \(\lambda\) 
		in every cell \((i,j)\) such that 
		\(c(x_i,y_j) \in  \{\beta_1,\ldots,\beta_m\}\). Thus, the support of 
		\(h\) is included in the support of \(h_c^m\), and $h_c^m$
		is a plan, so 
		\(m_c(\lambda)\leq m\) and 
		\[
		d_c(R_{\lambda}) = \lambda + \beta_{m_c(\lambda)} \leq \lambda + \beta_m = d_c(R_{\lambda}).  
		\]
		On the other hand, since \(h_c^{m_c(\lambda)}\) is a plan, we must have
		\[
		d_c(R_{\lambda}) \leq \lambda + \beta_{m_c(\lambda)}.  
		\]
		Combining the last two inequalities, we obtain that
		\(d_c(R_{\lambda}) = \lambda + \beta_{m_c(\lambda)}\), as desired, 
		and \(m=m_c(\lambda)\), so the support of \(h\) is included 
		in the support of \(h_c^{m_c(\lambda)}\). This 
		means \(h_c^{m_c(\lambda)}\) is itself a minimizing plan,
		and the last assertion
		follows.
	\end{proof}
	We collect the preceding conclusions in the following:
	\begin{thm}
		Let \(\mu\in\mathcal{M}(X)\), \(\nu\in\mathcal{M}(Y)\), that is,
		discrete max-plus probability measures on \(X\) and \(Y\)
		respectively, namely: \(\mu=\max_{i=1}^{m}(k_i +  \delta_{x_i})\),
		\(\nu=\max_{j=1}^{n}(k_j+ \delta_{y_j})\), and let
		\(c:X\times Y\to [0,\infty)\) be a given cost function.
		To obtain an optimal tropical plan $h$ between \(\mu\)
		and \(\nu\), one
		considers for every
		\(\lambda\in \Lambda(\mu)\cup \Lambda(\nu)\) (i.~e.~for
		each distinct weight of either \(\mu\) and \(\nu\))
		the respective region $R_\lambda$ and a minimizing
		plan $h_\lambda$
		for each $R_\lambda$ (e.g.\ $h_\lambda :=h_c^{m_c(\lambda)}$),
		setting then $h\in \Pi(\mu,\nu)$ to be the plan whose restriction
		over each $R_\lambda$ coincides with \(h_{\lambda}\). Furthermore, 
		\begin{align*}
			d_c(\mu,\nu)
			=
			\max_{\lambda\in \Lambda(\mu)\cup \Lambda(\nu)}(\lambda 
			+ c(x_{i_{\lambda}},y_{j_{\lambda}})),
		\end{align*}
		where each
		$(i_{\lambda},j_{\lambda}) \in f^{-1}(m_c(\lambda))$,
		\(f\) standing for the numbering function
		defined by condition~\eqref{eq_def_fnum1}. 
		In particular, if all the weights \(k_i\) and \(l_j\) are distinct,
		except \(k_1 = l_1 = 0\), then
		\begin{align*} 
			d_c(\mu,\nu) = 
			\max\limits_{1\leq i\leq m} 
			\min_{j\leq p_i} (k_i + c(x_i,y_j))  
			\vee \max \limits_{1\leq j\leq n}  
			\min_{i\leq q_j} (l_j + c(x_i,y_j)).
		\end{align*}
	\end{thm}
	
	\begin{proof}
		It is a direct consequence of combining 
		Lemma~\ref{lemma:supportsub} with 
		Proposition \ref{prop:decompose}.
	\end{proof}
	
	\subsection{Remarks on uniqueness of plans on a region}
	As we see fom Example~\ref{example:c1cs},,
	the function \(h_c^{m_c(\lambda)}\) (i.e.~the first function
	on \(R_{\lambda}\),
	as we go from 
	\(m=1\) to \(m=s\), that happens to be a plan) is 
	\emph{not} necessarily a \emph{reduced} plan.
	Another, simpler, example of such a situation is
	\[
	[c(x_i,y_j)]_{i,j=1}^2 = 
	\begin{pmatrix}
		1 & 3
		\\
		3 & 3 
	\end{pmatrix};
	\]
	indeed, supposing \(\{(1,1),(1,2),(2,1),(2,2)\}\)
	is a region \(R_{\lambda}\),
	then, here, \(m_c(\lambda) = 2\), and \(h^{m_c(\lambda)}\) is the
	\(2\times 2\) matrix with \(\lambda\) in every entry.
	
	We can state the following about reduced minimized plans and uniqueness of minimizing plans of a region.
	
	\begin{prop}\label{prop:uniquereduced}
		Let \(R_{\lambda}\) be a \(\lambda\)-region (corresponding to 
		some discrete max-plus probability measures \(\mu\) and \(\nu\)),  
		\(c:X\times Y\to [0,\infty)\) be a cost function.
		If \(h^{m_c(\lambda)}_c\) is a reduced plan, then 
		it is the \emph{unique} reduced minimizing plan for \(d_c(R_{\lambda})\).
		Vice versa, if a minimizing plan for  \(d_c(R_{\lambda})\) contains only
		\(-\infty\) and \(\lambda\) and is unique among minimizing plans
		with this property, then 
		it is reduced and must coincide with \(h^{m_c(\lambda)}_c\).
	\end{prop}
	
	\begin{proof}
		To prove the first assertion,
		suppose that \(h^{m_c(\lambda)}_c\) is a reduced plan for
		\(d_c(R_{\lambda})\).
		It is minimizing by Lemma~\ref{lemma:supportsub}.
		If there is another reduced minimizing plan \(h\) for 
		\(d_c(R_{\lambda})\), then by Lemma~\ref{lemma:supportsub} its 
		support is a subset of 
		the support of \(h_c^{m_c(\lambda)}\). Hence if 
		\(h\neq h_c^{m_c(\lambda)}\), then, for some \((x_i,y_j)\) 
		one has \(h(x_i,y_j) = -\infty\) and 
		\(h_c^{m_c(\lambda)}(x_i,y_j) =
		\lambda\). But, \(h^{m_c(\lambda)}_c\) being a reduced plan
		(by assumption), either the \(i\)-th row of the matrix 
		\([h_c^{m_c(\lambda)}(x_k,y_l)]_{k,l}\), or its \(j\)-th column,
		contain only \(-\infty\), except at \((i,j)\) where \(\lambda\) is.
		Therefore, the matrix \([h(x_k,y_l)]_{k,l}\)
		has either all the $i$-th column or all the $j$-th row full of
		$-\infty$, contradicting the fact that $h$ is a plan
		for $R_\lambda$, hence proving the assertion. 
		
		To prove the second assertion, let \(h\) be the unique minimizing plan for  \(d_c(R_{\lambda})\) 
		among minimizing plans containing only
		\(-\infty\) and \(\lambda\). 
		It has to be reduced by Lemma~\ref{lemma:red}. On the other hand,
		also \(h^{m_c(\lambda)}_c\) contains only
		\(-\infty\) and \(\lambda\) and is a minimizing plan
		for \(d_c(R_{\lambda})\), by Lemma~\ref{lemma:supportsub}.
		Thus \(h=h^{m_c(\lambda)}_c\) as claimed.
	\end{proof}
	
	We remark that the latter Proposition~\ref{prop:uniquereduced} asserts that
	having a unique plan (among all plans containing only 
	\(-\infty\) and \(\lambda\)) is equivalent to 
	\(h_c^{m_c(\lambda)}\) being reduced, but this is
	\emph{not} equivalent to the existence of a unique
	reduced minimizing plan as the following example shows. 
	\begin{example}\label{ex:distinctnotionsuniqueness}
		Suppose \(\lambda = 0\).
		\begin{enumerate}
			\item If the cost function is
			\begin{align*}
				[c(x_i,y_j)]_{i,j=1}^2 = 
				\begin{pmatrix}
					1 & 2 
					\\
					4 & 3 
				\end{pmatrix},
			\end{align*}
			then \(h_c^{m_c(\lambda)}\) is not reduced;
			there are two minimizing plans
			(containing only 
			\(0\) and \(-\infty\)), with one of them the only
			reduced minimizing plan: 
			\begin{align*}
				h_c^{m_c(\lambda)}=
				\begin{pmatrix}
					0 & 0
					\\
					-\infty & 0
				\end{pmatrix},
				\quad 
				h_1 = \begin{pmatrix}
					0 & -\infty
					\\
					-\infty & 0
				\end{pmatrix}.
			\end{align*}
			\item If the cost function is
			\begin{align*}
				[c(x_i,y_j)]_{i,j=1}^3 = 
				\begin{pmatrix}
					1 & 4 & 2 
					\\
					6 & 7 & 8
					\\
					5 & 9 & 3 
				\end{pmatrix},
			\end{align*}
			then \(h_c^{m_c(\lambda)}\) is not reduced,
			and there are at least two reduced minimizing plans:
			{\footnotesize 
				\begin{align*}
					h_c^{m_c(\lambda)}=
					\begin{pmatrix}
						0 & 0 & 0
						\\
						0 & -\infty & -\infty
						\\
						0 & -\infty & 0
					\end{pmatrix},
					\quad 
					h_1 = \begin{pmatrix}
						-\infty & 0 & 0
						\\
						0 & -\infty & -\infty 
						\\
						0 & -\infty & -\infty 
					\end{pmatrix},
					\quad 
					h_2 = \begin{pmatrix}
						-\infty & 0 & -\infty
						\\
						0 & -\infty & -\infty 
						\\
						-\infty & -\infty & 0 
					\end{pmatrix}.
				\end{align*}
			}
		\end{enumerate}
	\end{example}
	
	\subsection{A remark on perfect matchings}\label{subsection:remarkonpm}
	
	Of particular interest, as in the classical mass transportation problem, are minimizing  plans
	supported on subsets of the type
	\(\{(x_1,y_{\sigma(1)})\),
	\ldots,
	\((x_n,y_{\sigma(n)})\}\),
	where $\sigma\colon \{1.\ldots, n\} \to \{1.\ldots, n\}$.  
	We will call them
	\textit{perfect matching} plans.
	The plan $h_1$ in Example~\ref{ex:distinctnotionsuniqueness}(1) and
	the plan $h_3$ in Example~\ref{ex:distinctnotionsuniqueness}(2)
	are perfect matchings, while the other plans in these examples are not. 	
	The example below shows that for some data one might have
	no perfect matching minimizing plans. 
	\begin{example}\label{example:ultrametricnotenough} 
		Consider the cost matrix
		\[
		[c(x_i,y_j)]_{i,j=1}^3 = 
		\begin{pmatrix}
			5 & 1 & 5 
			\\
			5  & 2 & 5
			\\
			3 & 5 & 4
		\end{pmatrix}.
		\]
		If
		\(k_3 = k_2 = k_1 = l_3 = l_2 = l_1 = 0\),
		then 
		\[
		h 
		= 
		\begin{pmatrix}
			-\infty & 0 & -\infty 
			\\
			-\infty & 0 & -\infty
			\\
			0 & -\infty  &  0 
		\end{pmatrix}
		\]
		is  the unique minimizing plan
		(among plans containing only \(0\) and
		\(-\infty\)), but 
		is not a perfect matching. 
	\end{example}
	
	We stress that the nonexistence
	of the optimal tropical plans even when the max-plus probability measures
	$\mu$ and $\nu$ have all the weights equal to zero
	(as we said earlier, we call this case \textit{fundamental})
	is in striking contrast with the classical optimal mass transportation.
	The latter always admits an optimal transport plan corresponding
	to a perfect matching (i.~e.~a permutation matrix)
	between discrete measures which are sums of Dirac masses
	with equal weights, by virtue of the Birkhoff-von Neumann
	theorem which states that the set of extreme points of the
	Birkhoff polytope of bistochastic
	matrices in $\R^{n^2}$ is
	exactly the set of permutation matrices
	(and hence a linear functional on this polytope
	always attains its minimum on a permutation matrix). 
	
	The following assertion holds true.
	\begin{prop}\label{prop:whennopm}
		Let \(\mu = \max_{j=1}^{n} 
		(k_j + \delta_{x_j})\), 
		\(\nu = \max_{j=1}^{n} (l_j + \delta_{y_j})\),
		with the elements arranged as in \eqref{eq:standingassumption}
		as usual.
		If there is \(j\in\{1,\ldots,n\}\) such that \(k_j\neq l_j\),
		then there can be
		no plan that would correspond to a perfect matching.
	\end{prop}
	\begin{proof} 
		If \(h\in \Pi(\mu,\nu)\) is not reduced, then it 
		does not correspond to a perfect matching, so assume that
		\(h\in \Pi_R(\mu,\nu)\).
		Recall the 
		definition~\ref{eq:defofRlambda}
		and consider the disjoint regions
		\(R_{\lambda_k}\), $k=1,\ldots, r$
		determined by the plan \(h\), where
		$\lambda_k$, \(k=1,\ldots,r\)  
		are all the distinct weights 
		of the max-plus probability measures  \(\mu\) and  \(\nu\).
		Suppose that  the set
		\(
		\{i \colon   k_i=\lambda_k\} 
		\)
		has \(m_{k,1}\) elements, and the set 
		\(
		\{j \colon  \ l_j=\lambda_k\} 
		\)
		has \(m_{k,2}\) elements; at least one of these 
		two numbers must be positive. Observe that the plan \(h\)
		must have at least 
		\(\max\{m_{k,1},m_{k,2}\}\) finite (i.e.\ different from \(-\infty\))
		entries on the region \(R_{\lambda_k}\). Thus,
		the plan \(h\) has at least
		\[
		m = \max\{m_{1,1},m_{1,2}\} + \cdots + \max\{m_{r,1}, m_{r,2}\}  
		\]
		finite entries in total. 
		Keep in mind that
		\[
		\sum_{k=1}^r m_{k,1}  = \sum_{k=1}^r m_{k,2} =n.
		\]
		The plan will correspond to a perfect matching 
		only if there are \(n\) finite entries in total.
		The only way to have \(m=n\) is if
		\(m_{k,1}=m_{k,2}\) for every \(k=1,\ldots,r\).
		Given that the weights are arranged as in
		\eqref{eq:standingassumption} as usual, the conclusion follows.
	\end{proof}
	
	\section{Uniqueness of solution and perfect matchings for random costs}
	In this section, we will try to elucidate some questions
	regarding the optimal cost, perfect matchings and uniqueness when
	we introduce some randomness in the cost function.
	We will limit ourselves to the fundamental case
	(i.e.\ when all the weights of the discrete max-plus probability
	measures are zero) and with \(m=n\),
	i.~e.:
	\begin{align*}
		\muzerosn & = \max\{ 0 +\delta_{x_1},\ldots 
		0 + \delta_{x_m}\},\\
		\nuzerosn & = \max\{ 0 +\delta_{y_1},\ldots,
		0 + \delta_{y_n}\}.
	\end{align*}
	with  $x_j$, $j=1,\ldots, n$ as well as $y_i$, $i=1,\ldots, m$ all distinct.
	In what follows the sequences of max-plus probability measures $\muzerosn$
	and $\nuzerosn$ as above are fixed, while the cost function is random,
	i.~e.~is represented by a Bernoulli random matrix, i.~e.~ 
	each entry in the \(n\times n\) cost matrix is independent from the
	others and takes the value
	\(\beta_1\) with probability \(p\) and \(\beta_2\) with
	probability \(q=1-p\), where \(\beta_1< \beta_2\). 
	\subsection{Optimal tropical cost for random cost matrices}\label{subsection:tropicalcost}
	
	The following statement holds true.
	
	\begin{thm}\label{thm:bernoulli1}
			Let \(\beta_1,\beta_2\) be nonnegative numbers, with
			\(\beta_1<\beta_2\),
		and suppose that for each \(n\), 
		\(\muzerosn\) and \(\nuzerosn\)
		are discrete max-plus probability measures with all their weights
		equal to zero, and \(c^n\) is a Bernoulli cost matrix:
		\(\mathbb{P}(c^n(x_i,y_j)=\beta_1)=p\), 
		\(\mathbb{P}(c^n(x_i,y_j)=\beta_2)=q=1-p\) for
		\(i,j\in\{1,\ldots,n\}\), where
		\(x_1,\ldots,x_n\) and \(y_1,\ldots,y_n\) are the points
		of the support of \(\muzerosn\) and \(\nuzerosn\). If \(q<1\), then 
		\[
		\mathbb{P}(d_{c^n}(\muzerosn,\nuzerosn) = \beta_1) \to 1
		\quad \textrm{ as } \quad n\to \infty.
		\]
	\end{thm}
	\begin{proof}
		Even though a very short argument can be provided,
		we will derive a formula for the probability under question.
		Referring to
		Lemma \ref{lemma:supportsub} (and recall
		definition \ref{defn:mclambda})
		the optimal tropical cost \(d_{c^n}\)
		between
		\(\muzerosn
		=\max_{i=1}^n (0 + \delta_{x_i})\)
		and
		\(\nuzerosn =\max_{j=1}^n (0 + \delta_{y_j})\)
		will be \(\beta_1\) or \(\beta_2\) depending on whether
		\(m_{c^n}(0)\) is \(1\) or \(2\) respectively.
		It is \(1\) if and only if
		in the matrix for \(c^n\) there
		is at least one
		\(\beta_1\) in every row \emph{and} in every column.
		Denote by \(F_i\) the event that there is at
		least one \(\beta_1\) in the \(i\)-th row
		of the matrix, and by
		\(C_j\) the event that there is at least one \(\beta_1\)
		in the \(j\)-th column of the matrix.
		In the calculation that follows we 
		retain, for the sake of clarity, the notation \(m\) for the number
		of rows and
		\(n\) for the number of columns in the cost matrix, although one
		really has $m=n$. Therefore for the indices $i$ and $j$ one has 
		\(i\in\{1,\ldots,m\}\),
		\(j\in\{1,\ldots,n\}\).
		Thus
		\begin{align*}
			\mathbb{P}(d_{c^n}(\muzerosn,\nuzerosn) = \beta_1) 
			=  \mathbb{P}((\cap_{i=1}^m F_i) \cap (\cap_{j=1} C_j))
			= 1 - \mathbb{P}( (\cup_{i=1}^m F_i^c) \cup (\cup_{j=1} C_j^c)),
		\end{align*}
		where the upper index \(c\) denotes the complement
		of the event. We have
		\begin{align*}  
			\mathbb{P}( & (\cup_{i=1}^m F_i^c) \cup (\cup_{j=1}^n C_j^c)) 
			= 
			\\
			= & \ \sum_{s=1}^{m+n} (-1)^{s+1} 
			\sum_{\substack{a+b=s \\ (a,b)\neq (0,0)}}\binom{m}{a}\binom{n}{b}
			\mathbb{P}(F_1^c\cap \cdots \cap F_a^c \cap C_1^c\cap \cdots \cap C_b^c)
			\\
			= & \  \sum_{s=1}^{m+n} (-1)^{s+1} 
			\sum_{a+b=s}\binom{m}{a}\binom{n}{b}
			q^{mn-(m-a)(m-b)} 
			\\
			= & \ 
			- q^{mn} 
			\sum_{\substack{0\leq a \leq m \\ 0\leq b \leq n \\ (a,b)\neq (0,0)}} 
			(-1)^{a+b} \binom{m}{a}\binom{n}{b}
			q^{-(m-a)(m-b)}.
		\end{align*}
		Assuming that \(p<1\) (otherwise 
		\(\mathbb{P}(d_c(\muzeros,\nuzeros) = \beta_1) =1\) for
		any \(n\) so that there is nothing to prove).
		Then
		\begin{align*}
			\mathbb{P}( & (\cup_{i=1}^m F_i^c) \cup (\cup_{j=1}^n C_j^c))
			\\ 
			= & \ 
			-q^{mn}\bigg( 
			\sum_{\substack{0\leq a \leq m \\ 0\leq b \leq n }}
			(-1)^{a+b} \binom{m}{a}\binom{n}{b}
			q^{-(m-a)(m-b)} 
			-q^{-mn}
			\bigg)
			\\
			= & \ -q^{mn}(-1)^n \sum_{a=0}^m \binom{m}{a}(-1)^a
			\sum_{b=0}^n (-1)^{n-b}\binom{n}{b}(q^{-(m-a)})^{n-b} + 1 
			\\
			= & \ -q^{mn}(-1)^{m+n}\sum_{a=0}^b
			\binom{m}{a}(-1)^{m-a}(1-q^{-(m-a)})^n + 1.
		\end{align*}
		Recalling that \(m=n\), we get
		\begin{equation}\label{eq:formulac1}
			\mathbb{P}(d_{c^n}(\muzerosn,\nuzerosn) = \beta_1) 
			= q^{n^2}\sum_{j=0}^n (-1)^j \binom{n}{j}(1-q^{-j})^n.
		\end{equation}
		Thus,
		\[
		\mathbb{P}(d^n_c(\muzerosn,\nuzerosn) = \beta_1) \to 1
		\quad \textrm{ as } \quad n\to \infty,
		\]
		if \(q<1\), proving the claim.
	\end{proof}
	For the following remark, let us introduce a special notation for the expression
	in the right hand side of~\eqref{eq:formulac1}, namely, set
	\[
	\mathfrak{s}(n;p)
	:= 
	\begin{cases}
		(1-p)^{n^2}\sum_{j=0}^n (-1)^j 
		\binom{n}{j}(1-(1-p)^{-j})^n.
		&  \textrm{ if } p\in[0,1),
		\\
		1,  & \textrm{ if } p = 1.
	\end{cases}
	\]
	\begin{remark}\label{remark:sinprueba}
		The relationship~\eqref{eq:formulac1} reads
		\[
		\lim_{n}\mathfrak{s}(n;p) = 1 , 
		\quad 0 < p \leq 1.
		\]
		It is also easy to show that
		\[
		\lim\limits_{p\to 0} \mathfrak{s}(n;p) = 0,
		\quad 
		\lim\limits_{p\to 1} \mathfrak{s}(n;p) = 1,  
		\qquad n \in \mathbb{N},
		\]
		so that \(p\mapsto \mathfrak{s}(n;p)\) is continuous
		over $[0,1]$. 
		The asymptotics of $\mathfrak{s}$,
		hence that of the probability of the optimal tropical cost
		equaling the
		minimum value of the cost function,
		may be interesting also for the more general cases when
		$p$ is not constant but depends on $n$. 
		For instance, one has
		$\lim_{n\to\infty}\mathfrak{s}(n, 1/n^{\gamma}) = 0$ for all
		$\gamma \geq 1$ and 
		$\lim_{n\to\infty}\mathfrak{s}\big(n,1/n^{1/2}
		\big) = 1$.
	\end{remark}
	\begin{remark}\label{rem-binomial}
		A quite similar situation occurs not only when the cost is given not
		by a Bernoulli random matrix, but, say, by a binomial
		one.
		Namely, suppose now that \(s\in \N\) is fixed,
		and each entry in the cost matrix \(c^n\) can take  one
		of the values \(\beta_1<\cdots<\beta_s\) (as in
		\eqref{eq:c1cs}), with \(\beta_1\) appearing
		with probability \(p_1\). Let \(q:=1-p_1\).
		Then the lower bound for
		\(\mathbb{P}(d_{c^n}(\muzerosn,\nuzerosn) = \beta_1)\) can be
		obtained in the same way as in the proof of the
		Theorem~\ref{thm:bernoulli1}. Therefore
		\begin{align*}
			\lim_{n\to\infty}\mathbb{P}(d_c(\muzerosn,\nuzerosn) = \beta_1)
			= 1.
		\end{align*}
		Thus, even if the available choices for the entries
		of the cost matrix for \(c^n\) is a large but fixed number, 
		the optimal tropical cost between \(\muzerosn\)
		and \(\nuzerosn\) is equal to the the smallest value \(\beta_1\) of the cost
		with large probability for large $n$ (with probability of this event
		tending to one as  \(n\to\infty\)). 
		Moreover, if 
		\(p_j\) is the probability of \(\beta_j\) appearing
		in any given entry of the cost matrix, then it follows from
		the calculation above that
		\begin{equation}\label{eq:probofc1}
			\mathbb{P}\left(d_{c^n}(\muzerosn,\nuzerosn)=\beta_j\right)
			=
			\mathfrak{s}\left(n,\sum_{p=1}^j p_k\right) 
			- \mathfrak{s}\left(n,\sum_{p=1}^{j-1} p_k\right),
		\end{equation}
		which tends to zero as $n\to\infty$, the above equality~\eqref{eq:probofc1}
		giving the rate of convergence.
	\end{remark}
	
	\subsection{Presence of perfect matching optimal plans}
	We consider the following definition.
	\begin{defn}\label{defn:containsapm}
		Let \(\mu\) and \(\nu\) be discrete max-plus probability measures and 
		let \(h\) be a plan for a \emph{square} region \(R_{\lambda}\).
		We will say that \(h\) \emph{contains a perfect matching}, 
		if there is a
		perfect matching plan $\tilde h$ for the same region with support
		contained in the support of $h$.
	\end{defn}
	
	In other words, $h$ is a perfect matching plan for a region \(R_{\lambda}\)
	if it can be ``simplified'' by substituting some of its \(\lambda\)
	entries 
	by $-\infty$ to get a perfect matching plan  for a \(R_{\lambda}\).
	We will again discuss the case of a random cost provided by a Bernoulli
	cost matrix, and restrict 
	ourselves to the fundamental case. To simplify the discussion,
	let \(\beta_1 = 0\) and
	\(\beta_2 = 1\). If there is a zero in every row and
	every column of the matrix, then, as we know, the
	optimal tropical cost is \(0\), but if we look at the
	corresponding plan (represented by the matrix \(h\)),
	it may be impossible to ``simplify'' it
	(change some of the entries equal to $0$ to \(-\infty\))
	so as to produce a perfect matching plan (see 
	Example \ref{example:ultrametricnotenough}), that is,
	it does not contain a perfect matching.
	In the opposite
	direction, if the corresponding optimal plan contains a
	perfect matching, then
	the optimal tropical cost is \(0\). 
	Summing up, there are the following possibilities.
	\begin{itemize}
		\item The 
			optimal tropical cost
		is \(1\). 
		This occurs exactly when some row or column
		of the cost matrix
		fails to have a \(0\). Then there is always a perfect
		matching plan. 
		In fact, the absence of a \(0\) in some row or column
		of the cost matrix means that \(h_c^{m_c(0)}\) is the matrix
		with \(1\) in all
		the entries, which contains any perfect matching plan.
		For instance, if the cost matrix is
		\begin{align*}
			[c_{i,j}]_{i,j=1}^2 & = 
			\begin{pmatrix}
				0 & 1 
				\\
				1 & 1
			\end{pmatrix},
		\end{align*}
		then  a possible perfect matching minimizing plan is 
		\begin{align*}
			[h_{i,j}]_{i,j=1}^2 &=
			\begin{pmatrix}
				0 & -\infty
				\\
				-\infty & 0
			\end{pmatrix}.
		\end{align*}
		\item The 
			optimal tropical cost
		is \(0\), but the optimal plan
		does not contain a perfect matching.
		\item The 
			optimal tropical cost is \(0\), and the optimal
		plan contains a perfect matching.
	\end{itemize}
	For the following theorem we give here a random graph argument
	based on the strong and remarkable result of Bollob\'as
	and Thomason (see \cite[theorem~7.11]{bollobasrandomgraphs})
	that will also be
	used in the proof of Theorem \ref{thm:rarity1} below.
	\begin{thm}\label{th:pmtendstoone}
		Let \(\beta_1<\beta_2\), and suppose that for each natural
		number \(n\), 
		\(\muzerosn\) and \(\nuzerosn\)
		are discrete max-plus probability measures with all their weights
		equal to zero, and \(c^n\) is a Bernoulli cost matrix:
		\(\mathbb{P}(c^n(x_i,y_j)=\beta_1)=p_n\), 
		\(\mathbb{P}(c^n(x_i,y_j)=\beta_2)=q_n=1-p_n\) for
		\(i,j\in\{1,\ldots,n\}\), where
		\(x_1,\ldots,x_n\) and \(y_1,\ldots,y_n\) are the points
		of the support of \(\muzerosn\) and \(\nuzerosn\). 
		If \(p_n\geq (\log n)/n\) for all but finitely many \(n\), then
		\[
		\lim\limits_{n\to\infty} 
		\mathbb{P}(\exists h\in \Pi^{c^n}(\muzerosn,\nuzerosn) : 
		h \textrm{ contains a 
			perfect matching}) = 1.  	
		\] 			
	\end{thm}
	\begin{proof}
		We associate \(c^n\) with one and only one
		random bipartite (undirected) graph, denoted by \(G_n(c^n)\),
		with the sets \(\{x_1,\ldots,x_n\}\)
		and \(\{y_1,\ldots,y_n\}\) as the two disjoint sets 
		of vertices  
		in the following way: 
		\(c^n(x_i,y_j) = \beta_1\) if \(x_iy_j\) is an edge,
		and \(c^n(x_i,y_j) = \beta_2\) otherwise.
		The plan \(h_{c^n}^{m_{c^n}(0)}\) 
			(recall the definitions of
			section \ref{subsection:findingoptimalcost})
		contains a perfect matching
		plan if and only if the bipartite graph
		\(G_n(c^n)\) contains a perfect matching.
		In the proof of 
		\cite[theorem~7.11]{bollobasrandomgraphs}, it is shown that 
		the probability that the random bipartite graph contains
		a perfect matching approaches 1 as \(n\to\infty\).
		Thus, the probability that
		\(h_{c^n}^{m_{c^n}(0)}\) contains a perfect matching
		also approaches \(1\) as \(n\to\infty\). Since
		\(h_{c^n}^{m_{c^n}(0)}\) is always an optimal plan,
		the result follows.
	\end{proof}
	
	\begin{remark}
		An alternative proof of Theorem~\ref{th:pmtendstoone} can be
		offered
		as follows.
		Regardless of whether the optimal tropical cost is \(\beta_1\) or \(\beta_2\), 
		for the plan \(h_c^{m_c(0)}\) (which is always minimizing),
		the property of containing a perfect matching plan 
		is characterized by the fact that,
		for some permutation \(\sigma\in S_n\),
		the product  
		\[
		\Pi_{j=1}^n |\beta_2-c(x_j,y_{\sigma(j)})|  
		\] 
		is different from zero (necessarily then it is equal to $(\beta_2-\beta_1)^n$).
		The latter is guaranteed, for instance, when
		the matrix \([\beta_2-c(x_i,y_j)]_{i,j=1}^n\) is not
		singular (i.e.\ has nonzero determinant). By a theorem of
		Basak and Rudelson \cite{BASAK2017426},
		this probability approaches 1, for every
		\(0<p<1\).	
	\end{remark}
	
	\subsection{Uniqueness of minimizing plans}
	We show now that in the fundamental case
	(when all the weights of the discrete max-plus masure are zero),
	when the uniform probability
	is put on the space of the cost matrices, 
	the uniqueness of a
	minimizing plan 
	containing only \(0\) and \(-\infty\) 
	is an asymptotically rare event 
	in the sense that its
	probability tends to zero as the number of weights approaches infinity. 
	Namely, the following result is valid.
	
	\begin{thm}\label{thm:rarity1}
		Fix any positive real number \(M>0\) and 
		let \(\{X_n\}_{n=1}^{\infty}\) and
		\(\{Y_n\}_{n=1}^{\infty}\) be sequences
		of subsets of \(X\) and \(Y\) respectively,
		with \(\# X_n=\# Y_n=n\) for all \(n\). 
		For each \(n\in\N\), let
		\[
		\muzerosn := \max\{0+\delta_{x_1},\ldots,0+\delta_{x_n}\},
		\quad
		\nuzerosn := \max\{0+\delta_{y_1},\ldots,0+\delta_{y_n}\},\]
		where \(x_1,\ldots,x_n\) are the elements of \(X_n\)
		and \(y_1,\ldots,y_n\) those of \(Y_n\).
		For each \(n\)
		let \(P_n\) be the uniform 
		probability measure over
		\([0,M]^{X_n\times Y_n}\).
		Define \(C_n\subset [0,M]^{X_n\times Y_n} \)
		as the set of functions \(c\) such that there is 
		a unique, among plans containing only \(0\)
		and \(-\infty\), minimizing plan for 
		\(d_c(\mu_n,\nu_n)\).
		Then
		\[
		\lim\limits_{n\to\infty} P_n(C_n) = 0.
		\]
	\end{thm}
	
	\begin{proof}
		In order to apply the theory from
		\cite{bollobasrandomgraphs}, let us introduce 
		the notion of \emph{bipartite graph process},
		specifically, on the set of 
		vertices \(X_n \cup Y_n\). 
		Any given bijective function $f\colon \{1,\ldots, n^2\}\to \{1,\ldots, n\}^2$
		we define 
		determines a sequence of \(n^2+1\) graphs in the following way:
		at time step
		\(t=0\) there are no edges and at step \(t\in\{1,\ldots,n^2\}\)
		the edge 
		\((i,j):=f(t)\) is added.
		At the 
		\(n^2\)-th time step we obtain the complete bipartite graph.
		Note that the set of bijective functions
		$f\colon \{1,\ldots, n^2\}\to \{1,\ldots, n\}^2$
		is in one-to-one correspondence with the set of
		permutations of 
		$\{1,\ldots, n^2\}$, i.~e.~with the symmetric group 		
		\(S_{n^2}\) of order $n^2$; in fact, each \(f^{-1}\) is an
		enumeration of the cells of an $n\times n$ matrix.  
		If the function $f$ (or equivalently the respective
		permutation \(\sigma\in  S_{n^2}\)) is chosen randomly,
		with uniform
		probability, then we have a \emph{random bipartite
			graph process}, which coincides with the one described
		in~\cite{bollobasrandomgraphs} (see pp.~42 and~171 therein).
		Let \[
		\Omega_n := \{\omega: X_n\times Y_n \to [0,M] \ : 
		\ \omega \textrm{ takes } n^2 \textrm{ distinct values }\},\]
		and
		for each $\omega\in \Omega_n$ define
		the mapping
		$f_\omega\colon \{1,\ldots, n^2\}\to \{1,\ldots, n\}^2$
		by setting \(f_\omega(t):= (i,j)\),
		where  \((i,j)\)  is the unique pair of indices
		such that \(\omega(x_i,y_j)\) is the \(t\)-th largest
		value among the \(n^2\) distinct values
		\(\omega(x_1,y_1),\ldots,\omega(x_n,y_n)\).
		Thus, each \(\omega\in\Omega_n\) determines
		an ordering of the matrix cells \(f_{\omega}\)
		which, in turn,
		gives the above described graph process with \(f:=f_\omega\).
		
		Since \(P_n\) is the uniform measure on
		\([0,M]^{X_n\times Y_n}\), we have \(P_n(\Omega_n)=1\).
		Moreover, since \(P_n\) is uniform,
		for each bijective
		\(g \colon \{1,\ldots, n^2\} \to \{1,\ldots, n\}^2\),
		the set
		\(\{\omega\in \Omega_n  \colon f_\omega=g\}\)
		has the same \(P_n\)-measure, namely, \(1/(n^2)!\).
		Hence, these sets
		form a partition of the probability space
		\[([0,M]^{X_n\times Y_n},
		\mathcal{B}([0,M]^{X_n\times Y_n}),P_n)\]
		into \((n^2)!\) equiprobable events, where
		\(\mathcal{B}([0,M]^{X_n\times Y_n})\) stands for
		the Borel \(\sigma\)-algebra of \([0,M]^{X_n\times Y_n}\).
		Thus, the bipartite random graph process can be 
		equivalently sampled from this probability space, rather
		than directly from the set of bijective \(g \colon \{1,\ldots, n^2\} \to \{1,\ldots, n\}^2\) (or equivalently, from \(S_{n^2}\)) endowed with the
		uniform probability. Let us denote
		by \(\{G_t\}_{t=0}^{n^2}\)
		a generic realization of our bipartite random
		graph process on \(X_n\cup Y_n\), and let \(\tau\)
		be the stopping time
		\[\tau := \min\{t\ : \ G_t \textrm{ has degree }1\}.\]
		That is,
		\(\tau\) is the first  instance \(t\) such that every
		\(x_i\) belongs to an edge
		and also every \(y_j\) belongs to an edge.
		Recalling now
		Definition~\ref{defn:mclambda} and the algorithm
		of section 
		\ref{subsection:findingoptimalcost}, we have:
		\begin{equation}\label{eq:keylink}
			\tau(\omega) = m_{\omega}(0) 
		\end{equation}
		for \(P_n\)-a.e.~\(\omega\in \Omega_n\).
		Denote by \(D_n\) the event that
		\(G_{\tau}\) contains a perfect matching. By
		\cite[theorem.~7.11]{bollobasrandomgraphs},
		\begin{equation}\label{eq:bollobas}
			\lim_{n\to\infty} 
			\mathbb{P}_n(D_n)
			= 1,
		\end{equation}
		which means, in words, that by the time the bipartite 
		graph achieves
		degree 1 (this is exactly the time when the
		minimizing plan \(h_{\omega}^{m_{\omega}(0)}\) is formed, by
		\eqref{eq:keylink}),
		the graph contains a perfect matching. 
		Let 
		\[
		H_n := \{\omega\in \Omega_n \ : \ 
		h_{\omega}^{m_{\omega}(0)} \textrm{ is not reduced}\}.	
		\]
		By Proposition \ref{prop:uniquereduced}, we will
		be done if we show that \(P_n(H_n)\to 1\) as 
		\(n\to\infty\). Now, the event
		\(D_n\) is the disjoint union of
		\(F_n\) and \(E_n\), where \(F_n\) is the event
		that \(G_{\tau}\) is \emph{exactly} a perfect matching,
		and \(E_n\) is the event that \(G_{\tau}\) has
		a perfect matching and at least one more edge.
		As can easily be argued, \(P_n(F_n)\to 0\) as
		\(n\to\infty\) (in fact, for $F_n$ to hold, at the last step of forming $G_\tau$ only one possibility 
		of forming an edge, or equivalently only one way of placing a zero in the respective row of the matrix, results in a perfect matching).  
		Thus, by \eqref{eq:bollobas},
		\(P_n(E_n)\to 1\) as \(n\to\infty\). On the other hand,
		the event \(E_n\) is included in \(H_n\):
		indeed, a graph in \(E_n\) corresponds to a plan
		in the support of which there is triple of indices,
		two of which are in the same column and two of which
		are in the same row, thereby
		violating Definition \ref{def_reducedplan1}.
		Therefore, \(\lim_{n\to\infty}P_n(H_n)=1\) hence concluding the proof.
	\end{proof}

	\bibliographystyle{plain}
	\bibliography{tropical-biblio}	

@book{bollobasrandomgraphs,
	author = {B. Bollob\'as},
	edition = {2},
	publisher = {Cambridge University Press},
	title = {Random Graphs},
	year = {2010}}

@article{BASAK2017426,
	author = {A. Basak and M. Rudelson},
	doi = {https://doi.org/10.1016/j.aim.2017.02.009},
	issn = {0001-8708},
	journal = {Advances in Mathematics},
	pages = {426-483},
	title = {Invertibility of sparse non-Hermitian matrices},
	url = {https://www.sciencedirect.com/science/article/pii/S0001870817300427},
	volume = {310},
	year = {2017}}

@book{kolokoltsov1997idempotent,
	title={Idempotent Analysis and Its Applications},
	author={Kolokoltsov, V. and Maslov, V.P.},
	isbn={9780792345091},
	lccn={97012167},
	series={Mathematics and Its Applications},
	url={https://books.google.it/books?id=3mzgt0a03mQC},
	year={1997},
	publisher={Springer Netherlands}
}

@book {Villani06omt,
	AUTHOR = {Villani, C.},
	TITLE = {Optimal transport: old and new},
	SERIES = {Grundlehren der mathematischen Wissenschaften},
	VOLUME = {338},
	PUBLISHER = {Springer-Verlag},
	ADDRESS = {Berlin},
	YEAR = {2008}
}

@article {GimGomor64-tsp,
	AUTHOR = {Gilmore, P. C. and Gomory, R. E.},
	TITLE = {Sequencing a one state-variable machine: {A} solvable case of
	the traveling salesman problem},
	JOURNAL = {Operations Res.},
	FJOURNAL = {Operations Research},
	VOLUME = {12},
	YEAR = {1964},
	PAGES = {655--679},
	ISSN = {0030-364X,1526-5463},
	MRCLASS = {90.30},
	MRNUMBER = {171620},
	MRREVIEWER = {H.\ F.\ Karreman},
	DOI = {10.1287/opre.12.5.655},
	URL = {https://doi.org/10.1287/opre.12.5.655},
}

@article {GarfGilb78-bottlenecktsp,
	AUTHOR = {Garfinkel, R. S. and Gilbert, K. C.},
	TITLE = {The bottleneck traveling salesman problem: algorithms and
	probabilistic analysis},
	JOURNAL = {J. Assoc. Comput. Mach.},
	FJOURNAL = {Journal of the Association for Computing Machinery},
	VOLUME = {25},
	YEAR = {1978},
	NUMBER = {3},
	PAGES = {435--448},
	ISSN = {0004-5411,1557-735X},
	MRCLASS = {90C10 (05C35 90B10)},
	MRNUMBER = {496641},
	MRREVIEWER = {M.\ M.\ Sys\l o},
	DOI = {10.1145/322077.322086},
	URL = {https://doi.org/10.1145/322077.322086},
}
\end{document}